\documentclass[11pt]{amsart}
\usepackage{tikz}
\usepackage{pgfplots}
\usepackage{amsthm}
\usepackage{amsxtra}
\usepackage{amssymb}
\usepackage{graphicx}
\usepackage{comment}
\usepackage{url}
\usepackage{color}
\usepackage{amsfonts}
\usepackage{textcomp}
\usepackage{perpage,hyperref,enumerate}

\usepackage{caption}
\usepackage{subcaption}
\DeclareMathAlphabet{\mathpzc}{OT1}{pzc}{m}{it}
\usepackage{lmodern}

\usepackage{sidecap}

\DeclareMathOperator*{\diam}{diam}

\DeclareMathOperator*{\supp}{supp}
\DeclareMathOperator*{\WF}{WF}

\DeclareMathOperator*{\loc}{loc}
\DeclareMathOperator*{\comp}{comp}

\DeclareMathOperator*{\tr}{tr}

\usepackage{chngcntr}

\graphicspath{{Figures/}}

\newtheorem{theorem}{Theorem}

\numberwithin{prop}{section}
\newtheorem{corol}{Corollary}
\numberwithin{corol}{section}

\newtheorem{lemma}{Lemma}
\numberwithin{lemma}{section}

\numberwithin{conjecture}{section}

\newenvironment{remarks}[1][]{\begin{remark}\begin{trivlist}
\item[\hskip \labelsep {\bfseries #1}]\end{trivlist}\begin{itemize}}{\end{itemize}\end{remark}}

{\theoremstyle{definition}
\newtheorem{defin}{Definition}
\numberwithin{defin}{section}
}
\numberwithin{figure}{section}

\renewcommand{\Re}{\mathop{\rm Re}\nolimits}
\renewcommand{\Im}{\mathop{\rm Im}\nolimits}

\newcommand{\quotient}[2]{{\left.\raisebox{.2em}{$#1$}\middle/\raisebox{-.2em}{$#2$}\right.}}

\newcommand{\Spec}{\operatorname{Spec}}
\newcommand{\bl}{\begin{flushleft}}
\newcommand{\el}{\end{flushleft}}
\newcommand{\br}{\begin{flushright}}
\newcommand{\ert}{\end{flushright}}
\newcommand{\bc}{\begin{center}}
\newcommand{\ec}{\end{center}}

\newcommand{\imply}{\Rightarrow}

\newcommand{\numList}{\begin{enumerate}}
\newcommand{\enumList}{\end{enumerate}}

\newcommand{\e}{\epsilon}

\newcommand{\re}{\mathbb{R}}

\newcommand{\la}{\langle}
\newcommand{\ra}{\rangle}

\newcommand{\mc}[1]{\mathcal{#1}}
\newcommand{\pO}{\partial\Omega}
\theoremstyle{remark}
\newtheorem{remark}{Remark}

\newcommand{\Cc}{C_c^\infty}
\renewcommand{\O}[1]{\mathpzc{O}_{#1}}

\newcommand{\RN}[1]{\textup{\uppercase\expandafter{\romannumeral#1}}
}
\title[]{A Quantitative Vainberg Method for Black Box Scattering}
\author{Jeffrey Galkowski}
\email{jeffrey.galkowski@stanford.edu}
\address{Mathematics Department, Stanford University, Stanford, 
CA , USA}

\begin{document}
\begin{abstract}
We give a quantitative version of Vainberg's method relating pole free regions to propagation of singularities for black box scatterers. In particular, we show that there is a logarithmic resonance free region near the real axis of size $\tau$ with polynomial bounds on the resolvent if and only if the wave propagator gains derivatives at rate $\tau$. Next we show that if there exist singularities in the wave trace at times tending to infinity which smooth at rate $\tau$, then there are resonances in logarithmic strips whose width is given by $\tau$. As our main application of these results, we give sharp bounds on the size of resonance free regions in scattering on geometrically nontrapping manifolds with conic points. Moreover, these bounds are generically optimal on exteriors of nontrapping polygonal domains. 
\end{abstract}

\maketitle
\section{Introduction}
Let $P$ be a self-adjoint compact perturbation of the Euclidean Laplacian, e.g. $-\Delta+V$ where $V\in L^\infty_{\comp}$, $-\Delta_g$ for some metric $g_{ij}=\delta_{ij}+h_{ij}$ with $h$ a compactly supported metric, $-\Delta_D$, the Dirichlet realization of the Laplacian on $\re^d\setminus \Omega$, etc.
In their seminal works Lax--Phillips \cite{Lax} and Vainberg \cite{Vain1} understood the relationship between propagation of singularities for the wave group $e^{it\sqrt{P}}$ and pole free regions near the real axis for the meromorphic continuation of the resolvent 
$$(P-\lambda^2)^{-1}:L^2_{\comp}\to L^2_{\loc}$$
from $\Im \lambda\gg 1$. This relationship was extended to `black box' perturbations in the work of Tang--Zworski \cite{TangZw}. In the work of Vainberg and Tang-Zworski the authors show that if $P$ is quantum non-trapping, that is
$$\chi e^{-it\sqrt{P}}\chi:L^2\to C^\infty,\quad\quad t>T,$$ 
then there is an arbitrarily large logarithmic resonance free region. On the other hand Baskin--Wunsch \cite{BaWu} work in the weakly non-trapping setting where for all $N>0$ there exists $T_N>0$ so that for all $s\in \re$ 
$$\chi e^{-it\sqrt{P}}\chi:H^s\to H^{s+N},\quad\quad t>T_N$$
and show that there exists a logarithmic resonance free region.  However, as far as the author is aware, there has been no quantitative description of the relationship between the rate of smoothing and the size of resonance free regions. The purpose of this paper is to demonstrate such a quantitative relationship. 

So that we may consider many different types of perturbation of the Euclidean Laplacian at once, we work in the black box setting originally developed by Sj\"ostrand--Zworski in \cite{SjoDist}. In particular, the results apply to scattering in the presence of conic points \cite{BaWu, FoWu}, by delta potentials \cite{Galk,GS}, by bounded obstacles etc. 

\subsection{Resonance free regions in the black box setting}
We now recall the notion of a \emph{black box Hamiltonian} as in \cite[Chapter 4]{ZwScat}.  Let $\mc{H}_{R_0}$ be a Hilbert space and consider the Hilbert space, $\mc{H}$, with orthogonal decompostion
$$\mc{H}:=\mc{H}_{R_0}\oplus L^2(\re^d\setminus B(0,R_0))$$
for some $R_0>0$. We assume that $P:\mc{H}\to \mc{H}$ is self-adjoint with domain $\mc{D}\subset\mc{H}$ satisfying 
\begin{gather*} 
1_{\re^d\setminus B(0,R_0)}\mc{D}=H^2(\re^d\setminus B(0,R_0)),\quad\quad 1_{\re^d\setminus B(0,R_0)}P=-\Delta|_{\re^d\setminus B(0,R_0))}\\
1_{\re^d\setminus B(0,R_0)}(P+i)^{-1}\text{ is compact},\quad\quad P\geq -C.
\end{gather*}
We denote by 
\begin{gather*}
\mc{B}_{\loc}:=\mc{B}\cap (\mc{H}_{R_0}\oplus H^2_{\loc}B(0,R_0)),\quad \quad \mc{B}_{\comp}:=\mc{B}\cap (\mc{H}_{R_0}\oplus H^2_{\comp}(\re^d\setminus B(0,R_0)))
\end{gather*}
where $\mc{B}\subset \mc{H}$. Next denote for $s\geq0$
$$\mc{D}^{s}:=\{u\in \mc{H}\,|\,(2C+P)^{(s-2)/2}u\in \mc{D}\},$$ 
so that $\mc{D}^0=\mc{H}$ and let 
\begin{equation}\label{eqn:residual}\mc{R}:=\left\{R\mid R:\mc{D}^{-\infty}\to \mc{D}^\infty_{\comp}\right\}\end{equation}
denote the class of residual operator.  Here, $\mc{D}^{-s}$ for $s\geq 0$ is defined by duality.

Under these hypotheses, 
$$R_P(\lambda):=(P-\lambda^2)^{-1}:\mc{H}_{\comp}\to \mc{H}_{\loc}$$ 
admits a meromorphic continuation from $\Im \lambda\gg 1$ to $\mathbb{C}$ when $d$ is odd and to the logarithmic cover of $\mathbb{C}\setminus\{0\}$ when $d$ is even (see for example \cite[Chapter 4]{ZwScat}). 
Let 
$$U(t):=\frac{\sin  t\sqrt{P}}{\sqrt{P}}$$
where we use the spectral theorem to define $U(t)$. In the context of black box Hamiltonians, when we write $\chi \in \Cc(\re^d)$, we will implicitly assume that $\chi \equiv 1$ on $B(0,R_0)$ and define
 $$\chi u:=\chi1_{\re^d\setminus B(0,R_0)}u\oplus 1_{B(0,R_0)}u.$$ 
 Here, if $u=(u_1,u_2)\in \mc{H}=\mc{H}_{R_0}\oplus H^2(\re^d\setminus B(0,R_0))$, 
 $$1_{B(0,R_0)}u:=(u_1,0)$$
 is the orthogonal projection onto $\mc{H}_{R_0}.$
 With this in mind we have the following quantitative version of the Vainberg method.
\begin{theorem}
\label{thm:main}
Let $P$ be a black box Hamiltonian. Fix $R_1>R_0$. 
Suppose that for all $N>0$, there exists $T_{N}>0$ so that for all $\chi\in \Cc(\re^d)$ with $\supp \chi \Subset B(0,R_1)$ and $s\geq 0$,
\begin{equation} 
\label{eqn:smoothing}\chi U(t)\chi:\mc{D}^{s}\to \mc{D}^{s+N+1},\quad\quad t\geq T_{N}.
\end{equation}
Let 
$$T_{R_1,N}=\inf\{T'_{N}>0\mid \eqref{eqn:smoothing}\text{ holds for }T_{N}\}$$
and
$$\bar{T}:=\lim_{N\to \infty}\frac{T_{R_1,N}}{N}=\inf_{N>0}\frac{T_{R_1,N}}{N}.$$
Then for all $\delta>0$, there exists $\lambda_0>0$ so that for all $\chi$ with $\supp \chi\Subset B(0,R_1)$
$$R_{\chi}(\lambda):=\chi R_P(\lambda)\chi$$
continues analytically from $\Im \lambda>0$ to the region 
$$\Im \lambda >\begin{cases}-(\bar{T}^{-1}-\delta)\log \Re \lambda,& |\Re \lambda|>\lambda_0,\,\bar{T}\neq0\\
-\delta^{-1}\log \Re \lambda,&|\Re\lambda|>\lambda_0,\,\bar{T}=0.
\end{cases}$$
Moreover, there exist $C_j,B>0$, 
\begin{equation*}
\label{eqn:resolveEst}\left\|R_{\chi}(\lambda)\right\|_{\mc{D}^j}\leq C_j|\lambda|^{j-1}e^{B|\Im \lambda|},\quad \quad j=0,1\end{equation*}
in this region.
\end{theorem}
\begin{remarks}
\item 
The sequence $0\leq T_{R_1,N}$ is subadditive (see Section \ref{sec:subadd}) so the limit $\bar{T}\geq 0$ exists and has 
$$\bar{T}=\inf_{N>0}\frac{T_{R_1,N}}{N}.$$
\item Notice also that if there exists $\chi \in \Cc(\re^d)$ so that for all $N>0$, there exists $T_N>0$ so that for $s\geq 0$
$$\chi U(t)\chi :\mc{D}^s\to \mc{D}^{s+N},\quad\quad t\geq T_N,$$
then the assumptions of Theorem \ref{thm:main} are satisfied for any $R_1>R_0$ so that $B(0,R_1)\subset \{|\chi|\geq c>0\}$ with $T_{R_1,N}=T_N$.
\item It is not hard to see using propagation of singularities on $\re^d$ that for $R_2\geq R_1$, $T_{R_2,N}\leq T_{R_1,N}+C_{R_1,R_2}$ where $C_{R_1,R_2}$ depends only on $R_1$ and $R_2$. Therefore, $\bar{T}$ is independent of the choice of $R_1$ and the existence of such a $T_{R_1,N}$ for some $R_1>R_0$ implies the existence for all $R_1>R_0$.
\end{remarks}

In fact, we also have the following converse theorem.
\begin{theorem}
\label{thm:mainConverse}
Suppose that for some $\chi\in \Cc(\re^d)$, $\lambda_0>0$, $L,C,T>0$,  $M>0$, $R_{\chi}(\lambda)$ continues analytically to the region 
$$\left\{|\Re \lambda|>\lambda_0\,,\,\,\Im \lambda >-L\log |\Re \lambda|\right\}$$
with the estimate 
\begin{equation}
\label{eqn:resEst1}\|R_\chi(\lambda)\|_{L^2\to\mc{D}^1}\leq C|\lambda|^{M}e^{T|\Im \lambda|}
\end{equation}
in this region. 

Then for all $N>0$, $s>0$, 
$$\chi U(t)\chi:\mc{D}^s\to \mc{D}^{s+N+1},\quad\quad t\geq T_N$$
where 
$$T_N=\frac{N+M+T+2}{L}$$
In particular, 
$$\bar{T}=\lim_{N\to \infty}\frac{T_N}{N}=\frac{1}{L}.$$
\end{theorem}
Combining Theorems \ref{thm:main} and \ref{thm:mainConverse} gives the following corollary.
\begin{corol}
Suppose that for some $\chi \in \Cc(\re^d)$, $\lambda_0>0$, $L,C,T,M>0$, $R_{\chi}(\lambda)$ continues analytically to the region 
$$\left\{|\Re \lambda|>\lambda_0\,,\,\,\Im \lambda >-L\log |\Re \lambda|\right\}$$
with the estimate 
$$\|R_\chi(\lambda)\|_{L^2\to\mc{D}^1}\leq C|\lambda|^{M}e^{B|\Im \lambda|}$$
in this region. 
Then for all $\delta>0$, there exists $\lambda_\delta,C_\delta, T_\delta>0$ so that in the region
$$\left\{|\Re \lambda|>\lambda_\delta\,,\,\,\Im \lambda >-(L-\delta)\log |\Re \lambda|\right\}$$
we have the estimate
$$\|R_\chi(\lambda)\|_{L^2\to\mc{D}^j}\leq C_\delta|\lambda|^{j-1}e^{B_\delta|\Im \lambda|},\quad\quad j=0,1$$
\end{corol}

\subsection{Existence of resonances in the black box setting}
We make assumptions (see for example \cite[Section 4.3]{ZwScat} or \cite{StefanovUpperBound} for $d$ odd and \cite{Vod2, Vod3} for $d$ even) so that the wave trace  is defined as a distribution and so that there are polynomial bounds on the number of resonances.
To do this, we introduce a reference operator $P^\sharp$ defined as follows. Let 
$$\mathbb{T}^d_{R_1}:=\quotient{\re^d}{R_1\mathbb{Z}^d},\quad\quad R_1>R_0$$
and
$$\mc{H}_{R_1}^\sharp:=\mc{H}_{R_0}\oplus L^2(\mathbb{T}^d_{R_1}\setminus B(0,R_0))$$
where we identify $B(0,R_0)$ with its projection onto the torus. Define 
\begin{equation}
\label{eqn:refDomain}\mc{D}^\sharp_{R_1}:=\left\{u\in \mc{H}_{R_1}^\sharp\mid \chi \in \Cc(B(0,R_1)),\chi \equiv 1\text{ near }B(0,R_0)\imply \chi u\in \mc{D},(1-\chi)u\in H^2(\mathbb{T}^d_{R_1})\right\}.
\end{equation}
Then $P^\sharp_{R_1}:\mc{D}_{R_1}^\sharp \to \mc{H}_{R_1}^\sharp$ and for any $\chi$ as in \eqref{eqn:refDomain}, 
$$P_{R_1}^\sharp u=P(\chi u)+(-\Delta)(1-\chi)u.$$
We assume that there exists $n^\sharp\geq d$ so that
\begin{equation}
\label{eqn:numResAssume}
\left|\Spec(P^\sharp)\cap[-r^2,r^2]\right|\leq C_0 r^{n^\sharp},\quad\quad r\geq 1.
\end{equation}

Now, let 
$$\Lambda:=\{z\mid R_P(\lambda)\text{ has a pole at }z\}.$$
Then under the assumption \eqref{eqn:numResAssume} we have for any $a>0$ 
$$\#\left\{\lambda \in \Lambda\,\mid\, |\lambda|\leq r, |\arg \lambda|\leq a \right\}\leq Cr^{n^\sharp}.$$

The following theorem is implicitly proved in \cite[Chapter 6]{Galk} and is a simple consequence of \cite[Theorem 10.1]{SjoLocalTrace} (see also  \cite[Theorem 1]{SjoZw}) together with the Poisson formulae of \cite{SjZwJFA,ZwPoisson1}.
For simplicity, we use the half-wave propagator
$$W(t):=e^{-it\sqrt{P}},\quad\quad W_0(t):=e^{-it\sqrt{-\Delta_{\re^d}}},$$
but all of our results work equally well when $W(t)$ and $W_0(t)$ are replaced by $U(t)$ and $U_0(t)$. 
\begin{theorem}
\label{thm:mainRes}
Let
$$N(r,\rho):=\#\left\{\lambda\in \Lambda\,\mid\, |\Re \lambda|\leq r,\,\Im \lambda \geq -\rho\log |\Re \lambda|\right\}.$$
and as in \cite{ZwPoisson1}
$$u(t):=\tr (W(t)-1_{\re^d\setminus B(0,R_0)}W_0(t)1_{\re^d\setminus B(0,R_0)})+\tr 1_{ B(0,R_0)}W_0(t)1_{ B(0,R_0)}.$$
Suppose there exists $\{ T_k\}_{k=1}^\infty$, so that $T_k\to \infty$, and for each $k$ there exist $N_k>0$, $C_k>0$ so that for all $\varphi\in \Cc(\re)$ supported sufficiently close to $T_k$ with $\varphi(T_k)=1$,
 $$|\widehat{\varphi u}(\tau)| \geq C_k|\tau|^{-N_k},\quad\quad |\tau|\geq 1.$$ 
Let
 $$\bar{T}:=\liminf_{k\to \infty}\frac{T_k}{N_k}.$$
 Then for all $\e,\delta>0$ small enough there exist $R_0>0$ and $c>0$ so that for $r>R_0$ 
$$cr^{1-\e}\leq\begin{cases}N(r,\bar{T}^{-1}+\delta)&0<\bar{T}<\infty\\
N(r,\delta^{-1})&\bar{T}=0\\
N(r,\delta)&\bar{T}=\infty.\end{cases}$$
\end{theorem}

\subsection{Applications to scattering on manifolds with conic points}
The main application of our results is to scattering on manifolds with conic points. In \cite{BaWu} and \cite{FoWu} the Baskin--Wunsch and Ford--Wunsch respectively analyze the singularities of the wave group and the wave trace. We use these results together with Theorems \ref{thm:main} and \ref{thm:mainRes} to give a generically optimal bound on the size of the resonance free region for scattering on a manifold with conic singularities. In particular, let $X$ be a non-compact manifold with conic singularities so that there exists $K\Subset X$ with $X\setminus K$ isometric to $\re^d\setminus B(0,R_0)$ for some $R_0>0$. Assume further that
\begin{enumerate}
\item\label{a1} $X$ is geometrically nontrapping
\item\label{a2} No three cone points are collinear
\item\label{a3} No two cone points are conjugate
\end{enumerate}
Condition \eqref{a1} above asserts that for each compact set $K\subset M$, there is a time $T\geq0$ such that classical particles starting at $x_0\in K$ have $x(t)\notin K$ for $t\geq T$. This assumption is non-generic, but it is the natural situation in which there are logarithmic resonance free regions. Indeed, in the presence of trapping, there are typically resonances much closer to the real axis \cite[Chapter 6]{ZwScat}. Conditions \eqref{a2} and \eqref{a3} above are generically satisfied for manifolds with cone points and impose respectively that no geometric geodesic hits three cone points and a certain transversality condition between manifolds associated to cone points (see Section~\ref{sec:conSetup} for precise versions of these assumptions).

Let $\{x_i\}_{i=1}^N$ be the cone points in $X$ and 
$$D_{\max}=\sup_{i,j}\sup\left\{ t\mid \text{there is a geometric geodesic of length $t$ connecting $x_i$ and $x_j$}\right\}.$$
Then under the assumptions \eqref{a1}-\eqref{a3},
\begin{theorem}
\label{thm:resFreeConic}
For all $\chi \in \Cc(X)$, and $\delta>0$, there exists $\lambda_0>0$, $C>0$, $B>0$ so that the cut-off resolvent
$$R_\chi(\lambda):=\chi(-\Delta_g-\lambda^2)^{-1}\chi:L^2\to L^2$$
can be analytically continued from $\Im \lambda>0$ to 
$$\left\{|\Re \lambda|>\lambda_0\,,\,\, \Im \lambda >-\left(\frac{d-1}{2D_{\max}}-\delta\right)\log |\Re \lambda|\right\}$$
and
$$\|R_{\chi}(\lambda)\|_{L^2\to L^2}\leq C|\lambda|^{-1}e^{B|\Im \lambda|}$$
in this region.
\end{theorem}
Moreover, this theorem is optimal. Suppose that in addition to the assumptions above, we impose the following generic property
\begin{enumerate}
\setcounter{enumi}{3}
\item  the length spectrum of closed diffractive geodesics consists of only simple, isolated points.
\end{enumerate}  
We need a few more definitions before stating our next theorem. We denote by $\mc{SD}$, the set of strictly diffractive geodesics (see Definition \ref{def:strictDiffrac}) and by $\mc{SD}^+\subset\mc{SD}$, the set of strictly diffractive geodesics whose diffraction coefficient, $\mc{D}_\gamma$, (See equation \eqref{eqn:dCoeff}) is nonzero. Finally, for a closed geodesic, let $\mc{N}_\gamma$ denote the number of cone points through which $\gamma$ diffracts. Here, if $\gamma$ ends (and hence begins) at a cone point, we count that point only once. Then define
$$D_{\max}^+:=\sup\left\{\left.\frac{t}{N}\right|\text{there exists $\gamma\in \mc{SD}^+$ closed with length $t$ and }\mc{N}_\gamma=N\right\}.$$
Let $\Lambda$ denote the set of poles of $R_\chi(\lambda)$. 
Under assumptions (1)-(4), we have the following theorem of Hillairet--Wunsch \cite{WunschTalk}.
\begin{theorem}[Hillairet--Wunsch \cite{WunschTalk}]
\label{thm:resConic}
Suppose that $D_{\max}^+\neq-\infty$ and let
$$N(r,\rho):=\#\left\{\lambda\in \Lambda\,\mid\, |\Re \lambda|\leq r,\,\Im \lambda \geq -\rho\log |\Re \lambda|\right\}.$$
Then for all $\e,\delta>0$ there exists $R_0>0$ and $C>0$ so that for $R>R_0$ 
$$N\left(r,\frac{d-1}{2D_{\max}^+}+\delta\right)>cr^{1-\e}.$$

In particular, there exists a sequence $\{\lambda_n\}_{n=1}^\infty\subset \Lambda$ so that 
$$\Re \lambda_n\to \infty,\quad \quad \frac{-\Im \lambda_n}{\log |\Re \lambda_n|}\to \frac{d-1}{2D^+_{\max}}.$$
\end{theorem}
Note that as in \cite{BaWu}, Theorems \ref{thm:resFreeConic} and \ref{thm:resConic} apply in the setting of scattering in the exterior of a polygonal domain in $\re^2$. In particular, 
\begin{corol}
\label{cor:poly}
If $X=\re^2\setminus\Omega$ is the exterior of a nontrapping polygon where no three vertices are colinear and $\Delta$ is the Dirichlet or Neumann extension of the Laplacian, then the results of Theorem \ref{thm:resFreeConic} hold for the resolvent on $X$. If in addition, the length spectrum is simple and discrete, then the results of Theorem \ref{thm:resConic} hold.
\end{corol}
A result of Hillairet \cite[Section 3.2]{HilCone}, shows that generically on Euclidean surfaces with cone points $\mc{SD}^+=\mc{SD}$ and so $D_{\max}^+=D_{\max}$. In particular, this holds on surfaces none of whose cone points have cone angles equal to $2\pi/k$ for some $k\in \mathbb{Z}_+.$ Thus, when applying Theorem \ref{thm:resConic} to obtain Corollary \ref{cor:poly}, we see that if none of the angles of the polygons are equal to $\pi/k$, then $D_{\max}=D_{\max}^+$ and hence the corollary gives matching bounds from above and below on the size of resonance free regions.

 More generally, the diffraction coefficient $\mc{D}_\gamma$ depends only on the structure of the links of the cones, $(Y_\alpha,h_\alpha)$, with which $\gamma$ diffracts. Thus, the result of Hillairet applies to any manifold with cone points whose links are circles. As alluded to in \cite{CheegTayl}, one expects that for generic $Y_\alpha$, $|\mc{D}_\gamma|\neq 0$ and hence that $D_{\max}^+=D_{\max}$. However, notice that the case $D_{\max}^+=-\infty$ may occur. For example, on Euclidean surfaces with cone points if all cone points have cone angle $2\pi/k$ then $D_{\max}^+=-\infty$.  

We conjecture that $\mc{SD}^+=\mc{SD}$ unless there is a cone point whose link is a circle of length $2\pi/k$ or a sphere of radius 1. However, we do not pursue that in this article and the author is not aware even of a proof that $\mc{SD}^+=\mc{SD}$ generically.

 \subsection{Application to scattering by delta potentials}
Although we will not present the proof in detail since the results are contained elsewhere, we point out that that one can also recover a weaker version of \cite[Theorem 5.2]{Galk} from the analysis in \cite[Chapter 6]{Galk}. In particular, 
\begin{theorem}
Let $\Omega \Subset \re^d$ be strictly convex with smooth boundary and $V\in C^\infty(\pO)$ (independent of $\lambda$). Then for all $\chi \in \Cc(\re^d)$ and $\delta>0$, there exists $\lambda_0>0$, $C,B>0$ so that the resolvent
$$R_{\chi}(\lambda):=\chi (-\Delta+\delta_{\pO}\otimes V-\lambda^2)\chi$$
has an analytic continuation to the region
$$\left\{|\Re \lambda|>\lambda_0\,,\,\,\Im \lambda>-\left(\frac{1}{D_{\Omega}}-\delta\right)\log |\Re \lambda|\right\}$$
where $D_{\Omega}:=\diam(\Omega)$
and 
$$\|R_\chi(\lambda)\|_{L^2\to L^2}\leq C|\lambda|^{-1}e^{B|\Im \lambda|}$$
in this region.
\end{theorem}
This result is also generically optimal by \cite[Theorem 6.1]{Galk} which we repeat here for the convenience of the reader. For a generic set of $\Omega$ and $V$, letting $\Lambda$ denote the set of poles of $R_\chi(\lambda)$, we have the following theorem.
\begin{theorem}
Let
$$N(r,\rho):=\#\left\{\lambda\in \Lambda\,\mid\, |\Re \lambda|\leq r,\,\Im \lambda \geq -\rho\log |\Re \lambda|\right\}.$$
Then for all $\e,\delta>0$ there exists $R_0>0$ and $C>0$ so that for $R>R_0$ 
$$N\left(r,\frac{1}{D_{\Omega}}+\delta\right)>cr^{1-\e}.$$

In particular, there exists a sequence $\{\lambda_n\}_{n=1}^\infty\subset \Lambda$ so that 
$$\Re \lambda_n\to \infty,\quad \quad \frac{-\Im \lambda_n}{\log |\Re \lambda_n|}\to \frac{1}{D_{\Omega}}.$$
\end{theorem}

\subsection{Organization of the paper}
The paper is organized as follows. In Sections \ref{sec:resFree}, \ref{sec:smoothing}, and \ref{sec:existence} we prove respectively Theorems \ref{thm:main}, \ref{thm:mainConverse}, and \ref{thm:mainRes}. We then give the applications to scattering on manifolds with conic points in Section \ref{sec:conic}.

\noindent
{\sc Acknowledgemnts.} The author would like to thank Andras Vasy and Maciej Zworski for encouragement and valuable suggestions. Thanks also to Jared Wunsch and Luc Hillairet for their comments and reading of an earlier version as well as discussion about Theorem \ref{thm:resConic} and to the anonymous referees for their careful reading and many helpful comments. The author is grateful to the National Science Foundation for support under the Mathematical Sciences Postdoctoral Research Fellowship  DMS-1502661.

\section{Proof of Theorem \ref{thm:main}}
\label{sec:resFree}
All of the essential elements are contained in \cite[Proposition 8]{BaWu} \cite[Section 4.6]{ZwScat}  \cite{Vain}. In order to prove Theorem \ref{thm:main}, one need only keep track of various constants in those arguments.

Let 
$$(\mc{F}_{t\to \lambda}f)(\lambda):=\int e^{-it\lambda}f(t)dt$$
denote the Fourier transform mapping $\mc{S}'\to\mc{S}'$. Then the inverse Fourier transform is given by 
$$(\mc{F}_{t\to \lambda}^{-1}g)(t):=\frac{1}{2\pi}\int e^{it\lambda}g(\lambda)d\lambda.$$

We will use \cite[Lemma 6]{BaWu}, repeated here for the convenience of the reader,
\begin{lemma}
\label{lem:FTest}
Suppose that $H_1$ and $H_2$ are Hilbert spaces and $N(t):H_1\to H_2$ is a family of bounded operators having $k$ continuous derivatives in $t$ when $t\in \re$ and being analytic in $t$ for $\Re t>T>0$, and equal to 0 on $t<0$. Suppose that there are constants $j_0$, $k\geq j_0+2$ and $C_j$ so that for $0\leq j\leq k$,
\begin{equation}
\label{eqn:derEst}\left\|\partial_t^jN(t)\right\|\leq C_j|t|^{j_0-j},\quad \quad\Re t>T.
\end{equation}
Then the operator
$$\check N(\lambda)=\mc{F}_{t\to \lambda}^{-1}N(t):H_1\to H_2,\quad\quad \Im \lambda>0$$
continues analytically to the domain $-\frac{3\pi}{2}<\arg\lambda <\frac{\pi}{2}$ and when $|\lambda|>1$ has
$$\|\check N\|\leq C_j|\lambda|^{-j}e^{T|\Im \lambda|},\quad \quad j=0,\dots k$$
\end{lemma}
\begin{proof}
Since $\check N$ grows at most polynomially in $t$ and is supported int $\{t\geq 0\}$, it depends analytically on $\Im \lambda\geq 0$. Moreover, for $0\leq j\leq k$,
$$\check{N}=(-i\lambda)^{-j}\int_0^\infty e^{i\lambda t}\partial_t^jN(t)dt.$$
Then, define the contours $\Gamma_{\pm}$ in the $t$- plane consisting of $[0,T]\cup [T,T\pm i\infty).$ Then when $j\geq j_0+2$, 
\begin{align}
\label{eqn:RePos}\check{N}&=(-i\lambda)^{-j}\int_{\Gamma_+}e^{i\lambda t}\partial_t^jN(t)dt&0<\arg \lambda <\frac{\pi}{2}\\
\label{eqn:ReNeg}
\check{N}&=(-i\lambda)^{-j}\int_{\Gamma_-}e^{i\lambda t}\partial_t^jN(t)dt&\frac{\pi}{2}<\arg \lambda <\pi
\end{align}
The equation \eqref{eqn:RePos} can be continued to $\Re \lambda \geq 0$ and \eqref{eqn:ReNeg} to $\Re \lambda<0$. The estimates follow easily from these equations.
\end{proof}

We now prove Theorem \ref{thm:main}
\begin{proof}
Fix $\chi \in \Cc(\re^d)$ with $\chi\equiv 1$ on $B(0,R_0)$, $\supp \chi \Subset B(0,R_1)$, $N>0$. Then by assumption, there exists $T_{R_1,N}$ so that for all $s\geq 0$
$$\chi U(t)\chi:\mc{D}^s\to \mc{D}^{s+N+1},\quad \quad t\geq T_{R_1,N}.$$

Let $\chi_i\in \Cc(B(0,R_1))$, $i=0,\dots 3$ with $\chi_1=\chi$ and $\chi_i\chi_{i+1}=\chi_{i+1}$. Next, define the cutoff $\rho(t,x)$ so that $\rho|_{B(0,R_0)}=\rho(t)$, $0\leq \rho\leq 1$, and 
$$\rho(t,x)=\begin{cases}1&t\leq |x|+T_{R_1,N}\\
0&t\geq |x|+T_{R_1,N}+R_0+1.
\end{cases}$$
Then, since 
$$\chi_0 U(t)\chi :\mc{D}^s\to \mc{D}^{s+N+1}$$
for $t\geq T_{R_1,N}$ and any singularities in $(1-\chi_0)U(t)\chi$ not having the same smoothing property must be outgoing, the propagation of singularities theorem for the wave equation on $\re^d$ implies that
$$(1-\rho)U(t)\chi:\mc{D}^s\to \mc{D}^{s+N},\quad\quad t\geq 0.$$ 
For $g\in L^2,$ consider $\rho U(t)\chi g$. Then, letting $\Box:=\partial_t^2+P$, and $\Box_0:=\partial_t^2-\Delta$,
$$\begin{cases}\Box\rho U(t)\chi g=[\Box,\rho]U(t)\chi g,\\
\rho U(0)\chi g=0,\quad\quad D_t\rho U(0)\chi g=\chi g.
\end{cases}
$$
\begin{figure}
\includegraphics[width=.9\textwidth]{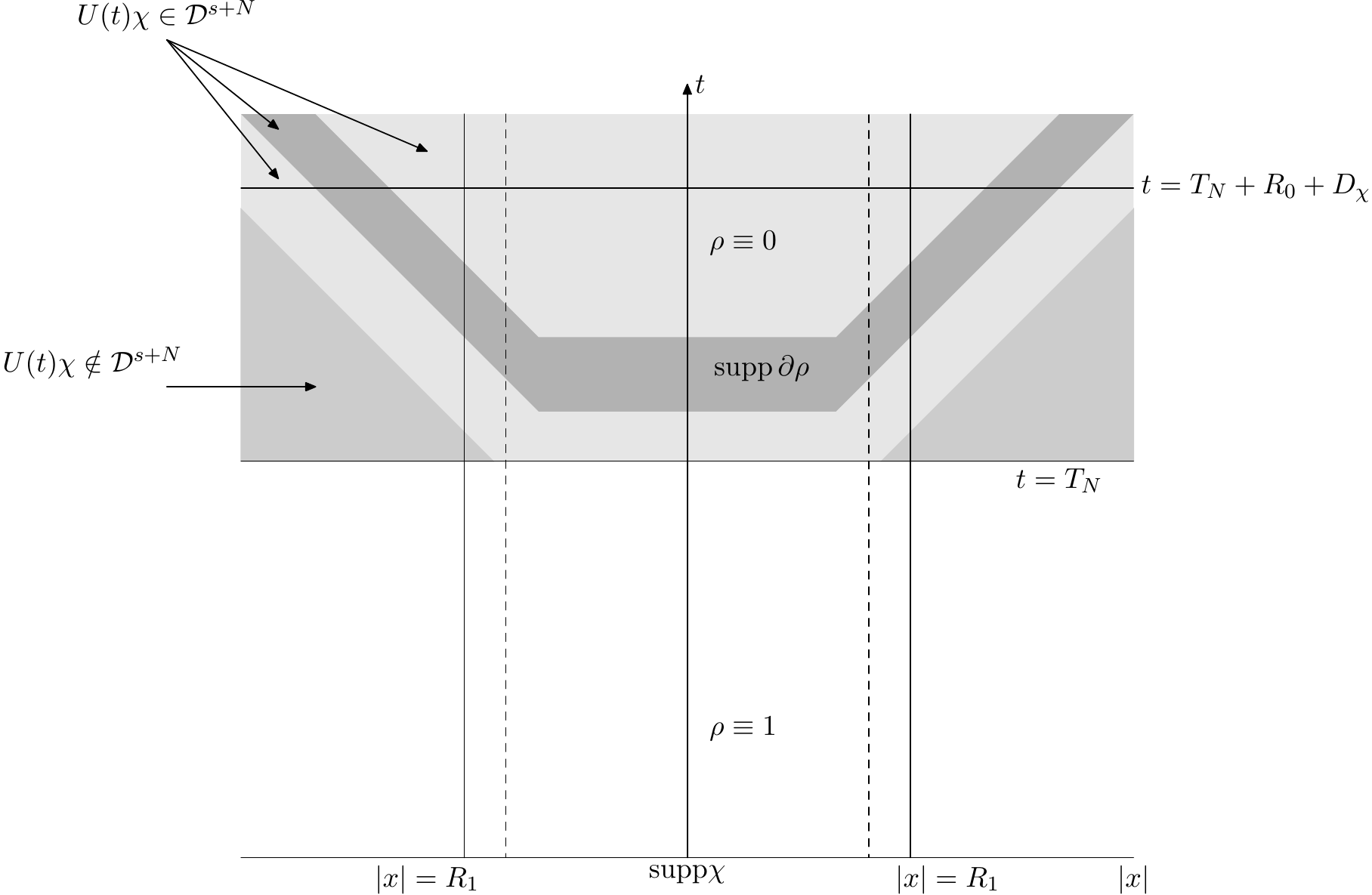}
\caption{We show support of $\chi$ and $\rho$ where $\supp\chi$ is given by the dashed box. We also show the propagation of $\mc{D}^{s+N}$ singularities for $t\geq T_N$.  }
\end{figure}
Then define 
$$F(t)g:=[\Box,\rho]U(t)\chi g=-\Box (1-\rho)U(t)\chi g.$$
 Notice that $F$ vanishes identically on $t<T_{R_1,N}$ and by our assumption
$$F(t)g\in C^0(\re_t;\mc{D}^{N-1})\cap C^{N-1}(\re_t;\mc{D}^0).$$
Moreover, $F(t)g$ has compact support in $t$ for any fixed $x$. In particular, 
$$\supp F(t)g\subset\{|x|+T_{R_1,N}\leq t\leq |x|+T_{R_1,N}+R_0+1\}.$$

Now, define an approximate resolvent 
$$R_1(\lambda)=\mc{F}_{t\to \lambda}^{-1}(\rho H(t)U(t)\chi),$$
where $H(t)$ denotes the Heavisde function. Then 
$$\partial_t^2[H(t)U(t)\chi g]=H(t)\partial_t^2U(t)\chi g+\delta(t)\chi g$$
and hence 
\begin{align*} 
(P-\lambda^2)R_1(\lambda)g&=\mc{F}^{-1}(\Box \rho H(t)U(t)\chi g)\\
&=\mc{F}^{-1}(H(t)\Box U(t)\chi g-\Box (1-\rho)H(t)U(t))+\chi g\\
&=\mc{F}^{-1}(F(t)g)(\lambda)+\chi g
\end{align*}
where we have used that $\supp F(t)\subset \{t\geq 0\}.$

Next, write 
$$Fg=\chi_2 F(t)g+(1-\chi_2)F(t)g=:F_1g+F_2g.$$
$R_1(\lambda)$ is not yet accurate enough to complete the proof. We must add a piece living only on $\re^d$. For this, let $W(t)$ denote the solution to 
$$\Box_0W(t)g=-F_2g,\quad\quad W(0)g=D_tW(0)g=0.$$
Observe that 
\begin{align*} 
-F_2g=\Box_0\chi_3W(t)g+\Box_0(1-\chi_3)W(t)g=-[\Delta,\chi_3]W(t)g+\Box_0(1-\chi_3)W(t)g
\end{align*}
Then, define a better approximation of the resolvent as 
$$R_2(\lambda)=R_1(\lambda)+\mc{F}_{t\to \lambda}^{-1}((1-\chi_3)W(t)).$$ 
Then since $P=-\Delta$ on $\supp (1-\chi_3)$, 
\begin{align*}
(P-\lambda^2)R_2(\lambda)g&=\chi g+\mc{F}^{-1}(F(t)g+\Box_0(1-\chi_3)W(t)g)\\
&=\chi\left(I+\mc{F}^{-1}(\chi_2F(t)+[\Delta,\chi_3]W(t))\right)g
\end{align*}
So, 
\begin{equation}\label{eqn:Resolve}R_2(\lambda)=R(\lambda)\chi\left(I+\mc{F}^{-1}(\chi_2F(t)+[\Delta,\chi_3]W(t))\right).
\end{equation}
We have the following estimates.
\begin{lemma}
\label{lem:ests}
Let ${T'_{R_1,N}}=T_{R_1,N}+R_0+2R_1+3$, then
\begin{align}
\left\|\chi R_1(\lambda)\right\|_{L^2\to \mc{D}^j}&\leq C_{j,1}|\lambda|^{j-1}e^{{T'_{R_1,N}}|\Im \lambda|},&j=0,1\label{eqn:est1}\\
\left\|\mc{F}^{-1}\chi F\right\|_{L^2\to L^2}&\leq C_{j,2}|\lambda|^{-j}e^{{T'_{R_1,N}}|\Im \lambda|},&j=0,1,\dots,N-1\label{eqn:est2}\\
\left\|\mc{F}^{-1}\chi(1-\chi_3)W\right\|_{L^2\to \mc{D}^j}&\leq C_{j,3}|\lambda|^{j-1}e^{{T'_{R_1,N}}|\Im \lambda|},&j=0,1\label{eqn:est3}\\
\left\|\mc{F}^{-1}[\Delta,\chi_3]W\right\|_{L^2\to L^2}&\leq C_{j,4}|\lambda|^{-j}e^{{T'_{R_1,N}}|\Im \lambda|},&j=0,1,\dots,N-1\label{eqn:est4}
\end{align}
\end{lemma}
\begin{remark}
We will also justify taking the Fourier transform in the process of proving Lemma \ref{lem:ests}.
\end{remark}
\begin{proof}
To obtain estimates \eqref{eqn:est1} and \eqref{eqn:est2}, we simply write the inverse Fourier transform,
$$\chi R_1(\lambda)=\frac{1}{2\pi}\int_0^{\infty} \chi \rho U(t)\chi e^{it\lambda}dt.$$
Then for $t>{T'_{R_1,N}}$, $\rho\chi=0$ and hence \eqref{eqn:est1} with $j=1$ follows from the energy estimate for the wave equation and with $j=0$ follows after an integration by parts in $t$. The estimate \eqref{eqn:est2} follows after observing that 
$$\supp\chi F\subset\{T_{R_1,N}\leq t\leq {T'_{R_1,N}}\}$$
and that we can integrate by parts up to $N-1$ times.

For the estimates \eqref{eqn:est3} and \eqref{eqn:est4}, we write
\begin{equation}
\label{eqn:W}W(t)=(1-\chi_2)\rho H(t)U(t)\chi g-H(t)U_0(t)(1-\chi_2)\chi g+q(t,z)
\end{equation}
where $U_0(t)=\sin t\sqrt{-\Delta}/\sqrt{-\Delta}$ is the free wave propagator. Then, 
\begin{align*}
\Box_0q&=\Box_0W-(1-\chi_2)\Box(\rho H(t)U(t)\chi g)+[\Box_0,\chi_2]\rho H(t)U(t)\chi g+\Box_0H(t)U_0(t)(1-\chi_2)\chi g\\
&=-F_2g-[\Delta,\chi_2]\rho H(t)U(t)\chi g-(1-\chi_2)(\Box(\rho U(t)\chi g)+\delta(t)\chi g)+\delta(t)(1-\chi_2)\chi g\\
&=-(1-\chi_2)Fg+(1-\chi_2)F-[\Delta,\chi_2]\rho H(t) U(t)\chi g\\
&=-[\Delta,\chi_2]\rho H(t) U(t)\chi g
\end{align*}
and
$$q(0,x)=D_tq(0,x)=0$$ 
Notice that 
\begin{equation}
\label{eqn:supprhs}\supp[\Delta,\chi_2]\rho H(t)U(t)\chi \subset (\re\times \supp \chi)\cap \{0\leq t\leq T_{R_1,N}+R_1+R_0+1\}.\end{equation}

Now, let $E_+(t)$ denote a forward fundamental solution for $\Box_0$ on $\re^d$. Then 
$$\chi (x) q(t,x)=-\int_0^t\chi(x)E_+(t-s)*[\Delta,\chi_2]\rho U(s)\chi gds.$$
We first observe that \eqref{eqn:supprhs} implies that $q(t,x)\equiv 0$ on $t\leq 0$. Moreover, for $d$ odd, the strong Huygens' principle implies that $\supp \chi q\subset \{t\leq {T'_{R_1,N}}-1\}.$

If $d$ is even, then we no longer have the strong Huygens' principle. However, for $t\geq {T'_{R_1,N}}-1$, the support of the right hand side is disjoint from the singular support of $\chi E_+$. Thus
$$\chi(x)q(t,x)=c_d\chi(x)\int_0^\infty\int_{\re^d}((t-s)^2+|x-y|^2)^{\frac{1-d}{2}}([\Delta,\chi_2]\rho U(s)\chi g(y))dyds.$$
In particular, $\chi q$ is analytic for $t>{T'_{R_1,N}}-1$ and has for some $j_0$,
\begin{equation}
\label{eqn:tempEst1}\left\|\partial_t^j \chi q\right\|_{\mc{D}^1}\leq C_j t^{j_0-j}\|g\|_{\mc{H}},\quad \quad t\geq {T'_{R_1,N}}-1,\quad j\geq 0.
\end{equation}
Similarly,
$$\chi U_0(t)(1-\chi_2)\chi g=\chi( E_+*(1-\chi_2)\chi g)$$
is analytic with polynomial bounds on its first derivative for $t>{T'_{R_1,N}}-1$. In particular, it satisfies for some $j_0$,
\begin{equation}\label{eqn:tempEst2}\left\|\partial_t^j \chi U_0(t)(1-\chi_2)\chi g\right\|_{\mc{D}^1}\leq C_j t^{j_0-j}\|g\|_{\mc{H}},\quad \quad t\geq {T'_{R_1,N}}-1,\quad j\geq 0.
\end{equation}
To see \eqref{eqn:est4}, observe that $[\Delta,\chi_3](1-\chi_2)\equiv 0$ so the first term in \eqref{eqn:W} vanishes. Moreover, $W$ has $N-1$ continuous derivatives as a map into $L^2$ and we have seen (by \eqref{eqn:tempEst1} and \eqref{eqn:tempEst2}) that the surviving terms in \eqref{eqn:W} satisfy \eqref{eqn:derEst} with $T={T'_{R_1,N}}$ and any $j$. Hence, \eqref{eqn:est4} holds by Lemma \ref{lem:FTest}.

Finally, we need to obtain \eqref{eqn:est3}. For this observe that the first term in \eqref{eqn:W} has the required estimate by \eqref{eqn:est1}. Next, consider 
$$B(t)g:=(1-\chi_3)(\chi q+\chi H(t)U_0(t)(1-\chi_2)\chi g).$$ 
Let $\psi\in \Cc(\re)$ have $\psi \equiv 1$ on $\{|t|\leq {T'_{R_1,N}}-1\}$ with $\supp \chi \subset\{|t|\leq {T'_{R_1,N}}\}.$ Then the estimate for $\mc{F}^{-1}((1-\psi(t))B(t))$ follows from Lemma \ref{lem:FTest}. The estimate for $\mc{F}^{-1}(\psi(t)B(t))$ follows similar to \eqref{eqn:est1} since $\partial_t^j B(t):\mc{H}\to \mc{D}^{1-j}$, $j=0,1$ are continuous on $t>0$.
\end{proof}

Now, we use estimates \eqref{eqn:est1} to \eqref{eqn:est4} in \eqref{eqn:Resolve} to complete the proof of Theorem \ref{thm:main}. The estimates \eqref{eqn:est1} and \eqref{eqn:est3} imply 
$$\|\chi R_2(\lambda)\|_{L^2\to \mc{D}^j}\leq C_j|\lambda|^{j-1}e^{{T'_{R_1,N}}|\Im \lambda|},\quad\quad j=0,1.$$
Fix $\delta,\delta_1>0$. To prove the required estimates on $\chi R(\lambda)\chi$, we only need to show that 
$$\|\mc{F}^{-1}(\chi_2F(t)+[\Delta,\chi_3]W(t))\|_{L^2\to L^2}<1-\delta_1.$$
For this, we simply look for $\lambda $ with $|\lambda|\geq 1$ and 
$$(C_{1,N-2}+C_{2,N-2})|\lambda|^{-N+1}e^{{T'_{R_1,N}}|\Im\lambda|}<1-\delta_1.$$
That is, for $|\lambda|\geq 1$ and
$$|\Im \lambda|<\frac{(N-1)\log |\lambda|-\log (C_{1,N-2}+C_{2,N-2})+\log(1-\delta_1)}{{T'_{R_1,N}}}.$$
\begin{equation}
\label{eqn:nearlyFinal}\|\chi R(\lambda)\chi\|_{L^2\to \mc{D}^j}\leq C_j\delta_1^{-1}|\lambda|^{j-1}e^{{T'_{R_1,N}}|\Im \lambda|},\quad\quad j=0,1.\end{equation}

Now, we can take $\log \lambda_0>\delta^{-1}(\log(C_{1,N-2}+C_{2,N-2})+\log(1-\delta_1))$ to obtain that on 
$$\left\{|\Re \lambda|>\lambda_0\,|\,|\Im \lambda|<\left(\frac{N-1}{T'_{R_1,N}}-\delta\right)\log |\Re \lambda|\right\}$$
\eqref{eqn:nearlyFinal} holds. 
Taking $N$ large gives the result since either $\bar{T}=0$ in which case $(N-1)/T'_{R_1,N}\to \infty$ or $\bar{T}>0$ in which case $(N-1)/{T'_{R_1,N}}\to \bar{T}^{-1}.$ 
\end{proof}
\subsection{Subadditivity of $T_{R_1,N}$}
\label{sec:subadd}
Let 
$$\mc{U}(t)=\begin{pmatrix} \partial_t U(t)&U(t)\\-PU(t)&\partial_t U(t)\end{pmatrix}.$$
Then, $\mc{U}(t)$ forms a one parameter group of operators.

Let $\chi_2\in \Cc(\re^d)$ (as above) have $\supp \chi \subset \{\chi_{2}\equiv 1\}$ and $\supp \chi_2\Subset B(0,R_1)$. Then, 
$$\mc{U}(t)\chi=(1-\chi_2)\mc{U}(t)\chi+\chi_2\mc{U}(t)\chi.$$
So, for $t>T_{R_1,N}$,
$$\chi \mc{U}(t)\chi =\chi \mc{U}(t-T_{R_1,N})(1-\chi_2)\mc{U}(T_{R_1,N})\chi+\chi \mc{U}(t-T_{R_1,N})\chi_2 \mc{U}(T_{R_1,N})\chi.$$
By propagation of singularities for the wave equation on $\re^d$ and the fact that 
$$\WF((1-\chi_2)\mc{U}(T_{R_1,N})\chi)\subset \{(t,z,\tau,\zeta)\mid |z|>R_0,\, \la z,\zeta/\tau\ra >0\}.$$
is outgoing the first term maps into $\mc{D}^\infty\oplus \mc{D}^\infty$ for $t>T_{R_1,N}$. For the second term, we see that for $t>T_{R_1,N}+T_{R_1,M}$, and $s\in \re$,
$$\chi_2\mc{U}(t-T_{R_1,N})\chi_1:\mc{D}^s\oplus \mc{D}^{s-1}\to \mc{D}^{s+M}\oplus\mc{D}^{s+M-1}.$$
Moreover, 
$$\chi_2\mc{U}(T_{R_1,N})\chi\in mc{D}^s\oplus \mc{D}^{s-1}\to \mc{D}^{s+N}\oplus\mc{D}^{s+N-1}$$ 
by assumption. 
Thus, for $t>T_{R_1,N}+T_{R_1,M}$,
$$\chi \mc{U}(t)\chi:mc{D}^s\oplus \mc{D}^{s-1}\to \mc{D}^{s+M+N}\oplus\mc{D}^{s+M+N-1}$$
and hence 
$$\chi U(t)\chi:\mc{D}^s\to \mc{D}^{s+N+M+1},\quad\quad t>T_{R_1,N}+T_{R_1,M}$$
and $T_{R_1,N}$ is subadditive.

\section{Proof of Theorem \ref{thm:mainConverse}}
\label{sec:smoothing}
We first prove a lemma that allows us to trade powers of $\lambda$ for regularity. Let 
$R_P(\lambda):=(P-\lambda^2)^{-1}$
as in the introduction.
\begin{lemma}
\label{lem:upgrade}
Suppose that 
$$\|\chi(R_P(\lambda)-R_P(-\lambda))\chi\|_{\mc{H}\to \mc{D}^j}\leq C\la \lambda\ra^{M+j}e^{T|\Im \lambda|},\quad j=0,1.$$
Then for all $s,r\geq 0$,
$$\|\chi(R_P(\lambda)-R_P(-\lambda))\chi\|_{\mc{D}^{s}\to\mc{D}^{s+r}}\leq C\la \lambda\ra^{M+r}e^{T|\Im \lambda|}.$$ 
\end{lemma}
\begin{proof}
First, observe that
\begin{align*}P\chi(R_P(\lambda)-R_P(-\lambda))\chi &=\lambda^2\chi(R_P(\lambda)-R_P(-\lambda))\chi -[\Delta,\chi]\chi(R_P(\lambda)-R_P(-\lambda))\chi 
\end{align*}
Notice that $[\Delta,\chi]:\mc{D}^s\to \mc{D}^{s-1}$ so, 
\begin{multline*}\|\chi(R_P(\lambda)-R_P(-\lambda))\chi \|_{\mc{D}^s\to \mc{D}^{r+s+2}}\leq C\la \lambda\ra^2\|\chi(R_P(\lambda)-R_P(-\lambda))\chi \|_{\mc{D}^s\to \mc{D}^{r+s}}\\+C\|\chi(R_P(\lambda)-R_P(-\lambda))\chi \|_{\mc{D}^s\to \mc{D}^{r+s+1}}.
\end{multline*}
Similarly,
\begin{align*}\chi(R_P(\lambda)-R_P(-\lambda))\chi P &=\lambda^2\chi(R_P(\lambda)-R_P(-\lambda))\chi -\chi(R_P(\lambda)-R_P(-\lambda))\chi [\Delta,\chi]
\end{align*}
so, 
\begin{multline*}
\|\chi(R_P(\lambda)-R_P(-\lambda))\chi \|_{\mc{D}^{s+2}\to \mc{D}^{r+s}}\leq C\la \lambda\ra^{-2}\left(\|\chi(R_P(\lambda)-R_P(-\lambda))\chi \|_{\mc{D}^s\to \mc{D}^{r+s}}\right.\\\left.+C\|\chi(R_P(\lambda)-R_P(-\lambda))\chi \|_{\mc{D}^{s+1}\to \mc{D}^{r+s}}\right).
\end{multline*}
The claim then follows by induction and interpolation.
\end{proof}
To prove Theorem \ref{thm:mainConverse}, we use an argument similar to that used to prove a resonance expansion in \cite[Theorem 2.7, 4.42]{ZwScat}, \cite{TangZw}.

First recall that by the spectral theorem together with the description of the spectrum of black box Hamiltonians, (see for example \cite[Theorem 4.5]{ZwScat}) and that by assumption $\chi R_P(\lambda)\chi$ is analytic in 
$$\{\Im \lambda\geq 0,\,\,|\Re\lambda|\geq R\}$$
we have
\begin{align*} 
U(t)\chi g&=\sum_{k=1}^K\frac{\sin t\mu_k}{\mu_k}v_k\la \chi g,v_k\ra+\int_0^\infty \frac{\sin t\sqrt{z}}{\sqrt{z}}dE_{z}(\chi g)\\
&=\sum_{k=1}^K\frac{\sin t\mu_k}{\mu_k}v_k\la \chi g,v_k\ra+\int_0^{R^2}\frac{\sin t\sqrt{z}}{\sqrt{z}}dE_{z}(\chi g)+\int_{R^2}^\infty \frac{\sin t\sqrt{z}}{\sqrt{z}}dE_{z}(\chi g)\\
&=\RN{1}+\RN{2}+\RN{3}
\end{align*}
where $\mu_k^2$ are the eigenvalues of $P$ with $v_k$ the corresponding normalized eigenfunctions. Here we choose $R>0$ so that there are no poles of $(P-\lambda^2)^{-1}$ on $\re \setminus (-R,R).$ This is possible by assumption. 
Now, notice that since $v_k\in \mc{D}^s$ for all $s$, 
$$\|\RN{1}\|_{\mc{D}^s}\leq C_N\|g\|_{\mc{D}^{-N}}.$$
Similarly, 
$$\int_0^{R^2}\frac{\sin t\sqrt{z}}{\sqrt{z}}dE_z:\mc{D}^{-N}\to \mc{D}^\infty$$
for all $N$ and $\chi:\mc{D}^s\to \mc{D}^s$ for all $s$. Therefore, 
$$\|\RN{2}\|_{\mc{D}^s}\leq C_N\|g\|_{\mc{D}^{-N}}$$
and we need only consider $\RN{3}$.

Now, using Stone's formula (see for example \cite[Theorem B.8]{ZwScat}), 
$$dE_z=\frac{1}{2\pi i}((P-z-i0)^{-1}-(P-z+i0)^{-1}) dz.$$
Then, making the change of variables, $z=\lambda^2$, we have 
$$dE_z=\frac{1}{\pi i}(R_P(\lambda+i0)-R_P(-\lambda +i0))\lambda d\lambda.$$
where $R_P(\lambda)=(P-\lambda^2)^{-1}$.
Then
\begin{align*} 
\RN{3}&=\frac{1}{2\pi }\lim_{\e\to 0^+}\int_R^\infty \left(e^{-it\lambda}-e^{it\lambda}\right)\chi (R_P(\lambda+i\e)-R_P(-\lambda+i\e))\chi gd\lambda \\
&=\frac{1}{2\pi}\int_{\Sigma_R} e^{-it\lambda}\chi(R_P(\lambda)-R_P(-\lambda))\chi gd\lambda
\end{align*}
where $\Sigma_R:=\re \setminus(-R,R)$. 

Let
$$I_\Sigma g:=\frac{1}{2\pi}\int_\Sigma e^{-it\lambda}\chi(R_P(\lambda)-R_P(-\lambda))\chi d\lambda.$$ 
In order to justify the convergence of $I_{\Sigma_\e}g$, we take $g\in \mc{D}^{M+2}$. Then Lemma \ref{lem:upgrade} implies the integral is norm convergent in $L^2$. We will be able to conclude using the fact that $\mc{D}^{M+2}$ is dense in $\mc{D}^s$ for all $s\leq M+2$. 

We deform the contour to $\Sigma_{R,\log}:=\gamma_{\pm}\cup \gamma_{\pm\log}$ where 
\begin{gather*} 
\gamma_{\pm}=\{\pm R-it\mid 0\leq t\leq L\log |R|\},\quad\quad\gamma_{\pm\log}=\{\pm t-iL\log |t|\mid R\leq t<\infty\}
\end{gather*}
Now, we have chosen $R>\lambda_0$ so that on 
$$\{|\Re \lambda|>\lambda_0\,,\,\,\Im \lambda >-L\log |\Re \lambda|\}$$
$$\|\chi R_P\chi \|\leq C|\lambda|^Me^{T|\Im \lambda|}.$$
Then by the norm convergence of the integral over $\Sigma_{R}$ we can deform the contour to 
$I_{\Sigma_{R}}g=I_{\Sigma_{R,\log}}g$.
We first estimate
\begin{gather*} 
\|I_{\gamma_\pm}g\|_{\mc{D}^{s+N}}\leq C\int_{0}^{L\log|R|} \la R\ra^{M+N}e^{T\lambda}\|g\|_{\mc{D}^s}d\lambda\leq C_R\|g\|_{\mc{D}^s}.
\end{gather*}
Finally, 
\begin{align*}
\|I_{\gamma_{\pm\log}}g\|_{\mc{D}^{s+N}}&\leq C\int_R^\infty e^{-tL\log \lambda}\lambda ^{M+N}e^{T\log \lambda}\|g\|_{\mc{D}^s}\leq C\int_R^\infty e^{(M+N+T-tL)\log \lambda}d\lambda\|g\|_{\mc{D}^s}
\end{align*}
This integral converges precisely when
$$t>\frac{M+N+T}{L}=T_N-\frac{1}{L}.$$
This completes the proof of the theorem when $s\geq M+2$ and for $s\leq M+2$, the density of $\mc{D}^{M+2}\subset \mc{D}^s$ completes the proof of the theorem.

\section{Lower bounds on the number of resonances in logarithmic regions}
\label{sec:existence}
We now prove Theorem \ref{thm:mainRes}. We have made assumptions so that wave trace of our problem, 
$$u(t)=\tr W(t)-W_0(t)\in\mc{D}'(\re)$$
is well defined.
\begin{remark}
Notice that this is not quite the actual definition of the wave trace since the operators act on different spaces. See the statement of Theorem \ref{thm:mainRes} for the precise formula.
\end{remark}
By \cite{SjoLocalTrace, SjZwJFA,ZwPoisson1}, we see that for $\varphi\in \Cc((0,\infty))$,
$$\varphi(t)u(t)= \sum_{\lambda\in \Lambda_\gamma} m(\lambda)e^{-i\lambda|t|}\varphi(t)+\O{C^\infty}(1)\,,\,\,$$
where 
$$\Lambda_\gamma:=\{\lambda \in \Lambda\mid \Im \lambda\geq -\gamma |\lambda|\}. $$ Moreover, by assumption, we have that for any $a>0$, there exists $C>0$ so that 
$$\#\{\lambda\in \Lambda\mid |\lambda |\leq r,\,|\arg \lambda|\leq a\}\leq Cr^{n^\sharp}.$$ 
By assumption, for all $\varphi\in \Cc((0,\infty))$ supported sufficiently close to $T_k$ with $\varphi(T_k)=1$, 
$$|\widehat{\varphi u}(\tau)|\geq C_k|\tau|^{-N_k}.$$
Therefore, by \cite[Theorem 10.1]{SjoLocalTrace} (see also  \cite[Theorem 1]{SjoZw}) for sufficiently small $\e>0$, $\delta>0$ and  any $\rho>(n^\sharp-N_k)/(T_k-\e^2)$, there exist $r_0>0$ such that 
$$N(r,\rho)>r^{1-\delta},\quad\quad r>r_0.$$
Now, letting $k\to \infty$ proves the theorem since $T_k\to \infty$.

\section{Distribution of resonances in scattering in the presence of conic points}
\label{sec:conic}
We now give the application of Theorems \ref{thm:main} and \ref{thm:mainRes} to scattering in the presence of conic singularities. For this, we recall the notation and results from \cite{BaWu}. 
\subsection{Geometric setup}
\label{sec:conSetup}
Let $X$ be a smooth noncompact manifold with boundary, $\partial X=Y$, $K$ a compact subset of $X$ and $g$ a Riemannian metric on $X^o$ such that $X\setminus K$ is isometric to $\re^d\setminus \overline{B(0,R_0)}$ for some $R_0>0$ and such that $g$ has conic singularities at the boundary of $X$ i.e.
$$g=dx^2+x^2h(x,dx,y,dy)$$
where $g$ is nondegenerate on $X^o$ and $h|_{\partial X}$ induces a metric on $\partial X$. We let $P=-\Delta_g$ be the Friedrichs extension of $-\Delta_g$ from $\Cc(X^o)$. Then $P$ is a black box Hamiltonian as described in the introduction and satisfies the assumptions of Theorem \ref{thm:mainRes}.

Let $Y_{\alpha}$, $\alpha =1,\dots N$ denote the connected components of $Y$. We call these the \emph{cone points of the manifold} since viewed in the manifold $X$ with metric $g$, they reduce to single points. Then, let $M:=\re\times X$ denote the spacetime manifold. 

Now, let $^X\mc{F}^s_\alpha$ denote the set of bicharacteristics in $T^*X^o$ whose continuations forward and backward in time reach $Y_\alpha$ in time $|t|\leq s$. It will sometimes be useful to refer to the incoming and outgoing parts of $^X\mc{F}^s_{\alpha}$ where the incoming and outgoing parts are given by the bicharacteristics whose forward, respectively backward, continuation reaches $Y_{\alpha}$. We write $\mc{F}^s_\alpha$ for the corresponding set in $T^*M^o$ i.e. the time $s$ flow out from the boundary $Y$. The manifolds $^X\mc{F}^s_\alpha$ and $\mc{F}^s_\alpha$ are \emph{coisotropic} manifolds respectively in $T^*X^o$ and $T^*M^o$. 

Next we define the notion of a diffractive geodesic. 
\begin{defin}
A \emph{diffractive geodesic} on $X$ is a union of a finite number of closed, oriented geodesic segments $\gamma_1,\dots, \gamma_N$ in $X$ such that:
\begin{enumerate}
\item all end points except possibly the initial point of $\gamma_1$ and the final point of $\gamma_N$ lie in $Y$
\item $\gamma_i$ ends at the same boundary component as $\gamma_{i+1}$ begins for $i=1,\dots , N-1$. 
\end{enumerate}
For such a geodesic, the number of cone points through which $\gamma$ diffracts is given by
$$\mc{N}_\gamma:=N$$
\end{defin}

Next, we define the notion of a geometric geodesic.
\begin{defin}
A \emph{geometric geodesic} on $X$ is a diffractive geodesic such that in addition the final point of $\gamma_i$ and the initial point of $\gamma_{i+1}$ for $i=1,\dots, N-1$ are connected by a geodesic of length $\pi$ in a boundary component $Y_{\alpha}$ with the metric $h_\alpha =h|_{Y_\alpha}.$ 
\end{defin}
\noindent Geometric geodesics are those that are locally realizable as limits of of families of geodesics in $X^o$ as they approach a given boundary component \cite[Proposition 1]{BaWu}. 
\begin{defin}
\label{def:strictDiffrac}
A diffractive geodesic is \emph{strictly diffractive} if for $i=1,\dots, N-1$, the final point of $\gamma_i$, $y_i$ and the initial point of $\gamma_{i+1}$ have $d_{h_{\alpha_i}}(y_i,y_{i+1})\neq \pi.$
\end{defin}
For each component of $Y$, $Y_{\alpha}$, let $-\Delta_{\alpha}$ denote the Laplacian on $Y_\alpha$ with respect to the metric $h_\alpha:=h|_{Y_{\alpha}}.$ Then define the operators as in \cite{FoWu},
$$\nu_\alpha:=\sqrt{-\Delta_\alpha+\left(\frac{2-d}{2}\right)^2},\quad\quad \mc{D}_\alpha:=e^{-i\pi \nu_\alpha}.$$
Then for a strictly diffractive geodesic, $\gamma$, let $y_i\in Y_{\alpha_i}$ be the final point of $\gamma_i$ and $x_i\in Y_{\alpha_i}$ be the initial point of $\gamma_{i+1}$. We define the \emph{diffraction coeffiction} of $\gamma$ by 
\begin{equation}\label{eqn:dCoeff}
\mc{D}_\gamma:=\prod_{i=1}^{N-1}|K_{\mc{D}_{\alpha_i}}(x_i,y_i)|
\end{equation}
where $K_{\mc{D}_{\alpha_i}}$ is the Schwartz kernel of $\mc{D}_{\alpha_i}.$

We can now write our assumptions for Theorem \ref{thm:resFreeConic} more precisely 
\begin{enumerate}
\item Let $\Omega \supset K$ be open with $X\setminus \Omega$ isometric to $\re^d\setminus B(0,R_1)$ for some $R_1>0$. We assume that there exists $T_0>0$ so that any geometric geodesic starting in $S^*_KX^o$ leaves $\Omega$ in time less than $T_0$. 
\item No geometric geodesic passes through three cone points.
\item No two cone points $Y_\alpha$ and $Y_\beta$ are conjugate to one another along geodesics in $S^*X^o$ of lengths less than $T_0$ in the sense that whenever $s+t\leq T_0$, $\mc{F}^s_\alpha$ intersects $\mc{F}^t_\beta$ transversally for each $\alpha, \beta$. 
\end{enumerate}
\begin{remark}
The third assumption is called non-conjugacy because by \cite[Proposition 6]{BaWu} the non-transversal intersection of $\mc{F}_{\alpha}^s$ and $\mc{F}_\beta^s$ is equivalent to the existence of a geodesic $\gamma$ with $\gamma(0)\in Y_\alpha$, $\gamma(t')\in Y_{\beta}$ with $(t'<s+t)$ and a normal Jacobi field $W$ along $\gamma$ so that $W$ is tangent to both $Y_\alpha$ and $Y_\beta$. 
\end{remark}
Throughout this section, it will be crucial to use the propagation of singularities theorem on manifolds with conic points due to Melrose--Wunsch \cite[Theorem 1.1]{MelWu}, originally observed in the case of product cones by Cheeger--Taylor \cite{CheegTayl}  (see also the more general setting of edge manifolds in Melrose--Vasy--Wunsch \cite{MelVaWu}). We state this theorem only informally and refer the reader to the original paper for the precise statement.
\begin{theorem}[Melrose--Wunsch \cite{MelWu}]
\label{thm:propSing}
Suppose that $u$ solves 
$$(\partial_t^2-\Delta_g )u=0.$$
Then $\WF(u)$ is contained in a union of maximally extended diffractive geodesics.
\end{theorem}

\subsection{Proof of Theorem \ref{thm:resFreeConic}}

Fix $\chi\in \Cc(X)$ with $\chi \equiv 1$ on $K$. Then let $K_1$ be a compact set containing $\supp \chi$. By assumption, there exists $T_0$ such that any geometric bicharacteristic starting in $\supp \chi$ leaves $K_1$ in time $T_0$. 

We now decompose $U(t)$ as in \cite[Section 3]{BaWu}. Let 
$$D_{\min}:=\min_{\alpha,\beta}d(Y_{\alpha},Y_\beta).$$ 
Then fix 
$\delta_A\ll 1$, $\delta_\psi\ll D_{\min}$
and let $\psi_\alpha\in \Cc(X)$ with 
$$\psi_\alpha\equiv 1\text{ on }\{d(x,Y_\alpha)<\delta_\psi/4\},\quad\quad\supp \psi_\alpha\subset \{d(x,Y_\alpha)<\delta_\psi\}. $$
 Next, let $\varphi \in C^\infty(X)$ equal to 1 outside of $\Omega$ and have $\varphi\equiv 0$ on $\supp \chi$. Finally, let $A_j$ $(j=1,\dots,N)$ be a pseudodifferential partition of $I-\sum\psi_\alpha-\varphi$ in which each $A_j$ has $\text{diam}(\WF(A_i))<\delta_A$ with respect to some metric on $S^*X$ and $\WF(A_i)\subset K_1$. In particular, so that we have 
$$I-\sum_\alpha \psi_\alpha -\varphi -\sum_iA_i\in \Psi_{\comp}^{-\infty}(X^o)$$
where $\Psi_{\comp}^{-\infty}$ denotes the set of compactly supported smoothing pseudodifferential operators.

We first consider the wave propagator precomposed with a cutoff away from the cone points. To do this, we decompose $U(t)$ into operators of the form
$$\mc{T}_J:=A_{j_0}U(t_0)A_{j_1}U(t_1)\dots A_{j_k}U(t_k)A_{j_{k+1}}$$
where $J=(j_0,\dots j_{k+1})$ is a word. We say that $J$ is \emph{diffractively realizeable} (DR) if there are points $p_l\in \WF(A_l)$, $p_{l+1}\in \WF(A_{l+1})$, $j=1,\dots k$ so that $p_{l}$ and $p_{l+1}$ are connected by a diffractive geodesic of length $t_l$. Similarly, we say that $J$ is \emph{geometrically realizeable} (GR) if there are points $p_l\in\WF(A_l)$ and $p_{l+1}\in \WF(A_{l+1})$, $j=1,\dots k$ so that $p_l$ and $p_{l+1}$ are connected by a geometric geodesic of length $t_l$.  

Recall the definition of the space $\mc{R}$ given in \eqref{eqn:residual}.
Then we have the following properties of $\mc{T}_J$ \cite[Section 3]{BaWu}
\begin{lemma}
\label{lem:propSingConsequence}
\begin{enumerate}
\item If the word $J$ is not DR then $\mc{T}_J\in \mc{R}.$
\item
There exists $\delta_A$ small enough so that if $t_0+t_1+\dots d_k>2T_0$, then the word $J=(j_0,\dots,j_{k+1})$ is not GR. \\
\item If $\delta_A$ is sufficiently small and $J=ijkl$ where $jk$ is GR and $ij$, $jk$, and $kl$ all interact with cone points then $\mc{T}_J\in \mc{R}.$
\end{enumerate}
\end{lemma}

It will be convenient to have notation for singularities that leave $K_1$ and never return. For this let 
$$\mc{O}:=\left\{(t,z,\tau,\zeta)\in T^*(M\setminus (\re\times \Omega))\mid \la z,\zeta/\tau\ra>0\right\}$$ 
be the outgoing set. It is not hard to see that the set $\mc{O}$ is mapped to itself by the positive time geodesic flow and any bicharacteristic starting in $\supp \chi$ that escapes $\Omega$ lies in $\mc{O}$ over $X\setminus \Omega$. We define $L^2H^s(\mc{O})$ to be the set of distributions whose wavefront set lies in $\mc{O}$ and lie in $L^2([0,T];H^s)$.
We now prove our main lemma which is a slight improvement on \cite[Lemma 5]{BaWu}. Let 
$$D_{\max}:=\max_{\alpha,\beta}\sup\{t\mid Y_\alpha,\,Y_\beta\text{ are connected by a geometric geodesic of length }t\}.$$ 
Note that $D_{\max}<\infty$ since it is clearly bounded by $T_0$. Thus, by \cite[Proposition 3]{BaWu}, $D_{\max}$ is achieved by some geometric geodesic connecting $Y_\alpha$ and $Y_\beta$ for some $\alpha, \beta$. 

\begin{lemma}
\label{lem:awayConic}
For all $N>0$, there exist $\delta_A$, $\delta_\psi$ sufficiently small so that for each $m$ and $s$ and all $t>(N+3)D_{\max}+2T_0+100=:T_N$
$$U(t)A_m:\mc{D}^s\to L^2\mc{D}^{s+N(d-1)/2-0}+L^2H^s(\mc{O}).$$
\end{lemma}
Here $u\in \mc{D}^{r-0}$, we mean that $u\in \mc{D}^{r-\e}$ for any $\e>0$. 
\begin{proof}
First note that if all diffractive geodesics starting from $\WF(A_m)$ leave $\Omega$ in time $t<2T_0+ND_{\max}$, then the result holds by Lemma \ref{lem:propSingConsequence}. 
\begin{remark}
Note that since $A_m$ is a pseudodifferential operator, we are abusing notation slightly and writing $\WF(A_m)\subset T^*X^o$ (rather than ${\WF}'(A_m)\subset T^*X^o\times T^*X^o$) for the wavefront set of such an operator.
\end{remark}
Therefore, we may assume this is not the case and hence that some geodesic hits a cone point within time $T_0$. Let $s_0$ denote the first time at which a cone point is reached from $\WF(A_m)$. Then in time $s_0+3\delta_\psi$, this bicharacteristic is at least $2\delta_{\psi}$ away from $\partial X$ (here we may take $\delta_\psi$ smaller if necessary). Therefore, taking $\delta_A$ small enough, and applying the propagation of singularities (Theorem \ref{thm:propSing}) we see that any singularity starting within $\delta_A$ of this one is propagated by $U(s_0+3\delta_\psi)$ to a distance greater than $\delta_\psi$ and less than $4\delta_\psi$ from the boundary and hence either $U(T_N)A_m$ has range in $H^s(\mc{O})$ or there exists $t_0<T_0$ such that 
$$U(	t _0)A_m=\sum_l A_lU(t_0)A_m\mod \mc{R}.$$
Now, we may remove all of the $A_l's$ so that $lm$ is not DR since they produce terms in $\mc{R}$. For those that are DR, we have seen that $d(Y_{\alpha},\WF(A_l))<4\delta_\psi.$ Now, repeating this argument, we have 
\begin{equation}
\label{eqn:propSum}U(T_N)A_m=\sum U(T_N-t_0-t_1-\dots-t_k)A_{j_{k+1}}U(t_k)A_{j_{k}}U(t_{k-1})\dots A_{j_1}U(t_0)A_m+E+R\end{equation}
where all of the words $J=(m,j_1,\dots, j_{k+1})$ are DR with 
$$\delta_\psi < d(Y_\alpha, \WF(A_{j_k}))<\sup\{d(Y_\alpha ,y)\mid y\in \WF(A_{j_k})\}<4\delta_\psi$$
for some $\alpha$, $E:\mc{D}^s\to H^s(\mc{O})$ and $R\in \mc{R}.$ 

Fix $\e>0$. Then for $\delta_\psi$ small enough each $t_i<D_{\max}+\e$. Indeed, if $t_i\geq D_{\max}+\e$, then there are two cases. Either $j_{i}j_{i+1}$. interacts with more than one cone point or all bicharacteristics starting in $\WF(A_{j_i})$ leave $\Omega$. By construction $j_ij_{i+1}$ interacts with only one cone point. Therefore, we must have that all bicharacteristics starting in $\WF(A_{j_i})$ leave $\Omega$ and hence terms like this can be absorbed into the operator $E$. 

Now, as in \cite{BaWu}, we associate a string of $D$'s and $G$'s to each word signifying a diffractive interaction and geometric interaction respectively. 
\begin{lemma}
Suppose that $J$ is a word containing containing $k+1$ consecutive $D'$s. Then 
 $$\mc{T}_J:\mc{D}^s\to \mc{D}^{s+k(d-1)/2-0}$$
 where we say that 
 $u\in \mc{D}^{s-0}$ if $u\in \mc{D}^{s-\e}$ for all $\e>0$.
\end{lemma}
\begin{proof}
For $k=1$, this lemma is proved in the course of the proof of \cite[Lemma 5]{BaWu}. Since proving it for $k>1$ involves only small adjustments, we omit the proof here.

\end{proof}

We now show that the sum \eqref{eqn:propSum} contains only words with $N+1$ consecutive $D$'s. To see this, first observe that by Lemma \ref{lem:propSingConsequence}, no strings of the form DGD, GGD, DGG, GGG occur in the sum. Therefore, a word of length $N+3$ occurring in \eqref{eqn:propSum} must contain at least $N+1$ consecutive $D$'s and all that remains to show is no words of length less than $N+3$ may occur. 

Suppose that there is a word $J$, in \eqref{eqn:propSum} so that $|J|=k+1<N+4$. Then, since each of $t_i<D_{\max}+\e$, 
$$T_N-t_0-t_1-\dots -t_{k}>2T_0+100+(N+3-k-1)D_{\max}-k\e>2T_0+100.$$
Hence, since singularities in $\mc{T}_J$ do not interact with another cone point (or there would be another letter in $J$),
$U(T_N-t_0-t_1-\dots -t_{k})\mc{T}_J$ maps singularities to $\mc{O}$.

This completes the proof of Lemma \ref{lem:awayConic} since every term in $U(T_N)A_m$ either sends singularities to $\mc{O}$ or smooths them by $N(d-1)/2-0$ derivatives.
\end{proof}
We now complete the proof of Theorem \ref{thm:resFreeConic} by showing
\begin{lemma}
Fix $N>0$ and let $T_N$ be as in Lemma \ref{lem:awayConic}. Then for all $t>T_N+1$ and $s\in \re$,
$$\chi U(t)\chi:\mc{D}^s\to \mc{D}^{s+N(d-1)/2-0}.$$ 
\end{lemma}
\noindent Note that then applying Theorem \ref{thm:main} gives Theorem \ref{thm:resFreeConic}.
\begin{proof}
First, write
\begin{equation}
\label{eqn:simpleDecompose} \chi U(t)\chi =\sum_j\chi U(t)\chi A_j+\sum_\alpha \chi U(t)\chi \psi_\alpha.
\end{equation}
Then the first term has the desired mapping properties by Lemma \ref{lem:awayConic}. To handle the second term, we pick 
$$2\delta_\psi<\tau<\min(D_{\min}/50,1).$$
Then by the propagation of singularities (Theorem \ref{thm:propSing}), if $\WF(v)\subset \WF(\psi_\alpha)$, then for all $\beta$,
$$\psi_\beta U(\tau)v\in \mc{D}^\infty.$$
Thus, 
$$U(\tau)v-\sum_jA_jU(\tau)v\in L^2\mc{D}^\infty +L^2\mc{D}^s(\mc{O})$$
and the second term in \eqref{eqn:simpleDecompose} has 
$$\sum_\alpha \sum_j\chi U(t-\tau)(A_jU(\tau)\chi \psi_\alpha)\mod \mc{R}.$$
Then, applying Lemma \ref{lem:awayConic} with $t>\tau +T_N$ completes the proof.
\end{proof}
\subsection{Proof of Theorem \ref{thm:resConic}}
To prove Theorem \ref{thm:resConic}, we simply apply \cite[Main Theorem]{FoWu} together with Theorem \ref{thm:mainRes}.
Fix $\delta>0$ and let $\e>0$ to be chosen small enough (depending on $\delta$). Since $D^+_{\max}\neq -\infty$, $0<D^+_{\max}\leq D_{\max}<\infty$ and there exists a closed strictly diffractive geodesic $\gamma\in \mc{SD}^+$ with $\mc{N}_\gamma=k\geq 2$, with length $l>k(D^+_{\max}-\e)$. Hence, there exists $\gamma_1\in \mc{SD}^+$ with $\mc{N}_\gamma=Nk$ and length $Nl$. Thus, by \cite[Main Theorem]{FoWu} for $\varphi\in \Cc(\re)$ supported sufficiently near $t=Nl$ with $\varphi(Nl)=1$, and $|\tau|\geq 1$,
$$|\widehat{\varphi(t)u(t)}(\tau)|\geq c|\tau|^{-Nk(d-1)/2}$$
where $u(t)$ is as in Theorem \ref{thm:mainRes}. Applying Theorem \ref{thm:mainRes} gives that for any $\e_1>0$ there exists $c>0$ such that 
$$N\left(r,\frac{d-1}{2(D^{+}_{\max}-\e)}+\e\right)\geq N\left(r,\frac{(d-1)k}{2l}+\e\right)\geq r^{1-\e_1}.$$
Hence, taking $\e$ small enough depending on $\delta$ proves Theorem \ref{thm:resConic}.
\bibliographystyle{abbrv} 
\bibliography{references.bib}
\end{document}